\documentclass[11pt]{amsart} 
\usepackage{xspace,amssymb,amsmath,amscd,amsthm,epsfig,etoolbox}

\newlength{\defbaselineskip}
\theoremstyle{plain}
\newtheorem{theorem}{Theorem}

\newtheorem{prop}[theorem]{Proposition}

\theoremstyle{definition}

\newcommand{\dqbin}[3]{\displaystyle\genfrac{[}{]}{0pt}{}{#1}{#2}_{#3}}
\newcommand{\tqbin}[3]{\textstyle\genfrac{[}{]}{0pt}{}{#1}{#2}_{#3}}

\newcommand{\Z}{\mathbb{Z}}

\newcommand{\mcC}{\mathcal{C}}
\newcommand{\mcR}{\mathcal{R}}

\DeclareMathOperator{\coinv}{coinv}
\DeclareMathOperator{\GF}{GF}
\DeclareMathOperator{\inv}{inv}

\DeclareMathOperator{\wt}{wt}

\begin{document}

\title{Bijective proofs of some coinversion identities related to
 Macdonald polynomials}
 
\author{Nicholas A. Loehr}
\thanks{This work was supported by a grant from the Simons
 Foundation/SFARI (\#633564 to Nicholas Loehr).} 

\begin{abstract}
This paper gives bijective proofs of some novel coinversion identities
first discovered by Ayyer, Mandelshtam, and Martin~\cite{amm-mac} as part of
their proof of a new combinatorial formula for the
modified Macdonald polynomials $\tilde{H}_{\mu}$.
Those authors used intricate algebraic manipulations of $q$-binomial 
coefficients to prove these identities, which imply the existence of certain 
bijections needed in their proof that their formula satisfies the axioms 
characterizing $\tilde{H}_{\mu}$. They posed the open problem of constructing 
such bijections explicitly. We resolve that problem here.
\end{abstract}

\maketitle

\section{Introduction}
\label{sec:intro} 

We begin by reviewing the coinversion statistic and its relation
to $q$-binomial coefficients and $q$-multinomial coefficients.
Given a formal variable $q$ and a positive
integer $n$, define the \emph{$q$-integer} $[n]_q=1+q+q^2+\cdots+q^{n-1}$
and the \emph{$q$-factorial} $[n]!_q=\prod_{j=1}^n [j]_q$. We also
set $[0]_q=0$ and $[0]!_q=1$.  For integers $k,n$ with $0\leq k\leq n$,
define the \emph{$q$-binomial coefficient} 
\[ \dqbin{n}{k}{q}=\frac{[n]!_q}{[k]!_q [n-k]!_q}. \]
We also define $\tqbin{n}{k}{q}=0$ when $k<0$ or $k>n$.
For nonnegative integers $k_1,k_2,\ldots,k_s$ with $k_1+k_2+\cdots+k_s=n$,
define the \emph{$q$-multinomial coefficient}
\[ \dqbin{n}{k_1,k_2,\ldots,k_s}{q}=\frac{[n]!_q}{[k_1]!_q [k_2]!_q\cdots
 [k_s]!_q}. \]
If any $k_i$ is negative, the $q$-multinomial coefficient is defined to be $0$.

For a word $w=w_1w_2\cdots w_n$ where each $w_i$ is an integer,
the \emph{coinversion statistic} $\coinv(w)$ is the number of pairs $(i,j)$
with $i<j$ and $w_i<w_j$.  For example, $\coinv(231132)=6$.
Let $\mcR(1^{k_1}2^{k_2}\cdots s^{k_s})$ be the
set of all words $w$ consisting of $k_1$ copies of $1$, $k_2$ copies of $2$,
and so on. The following combinatorial formulas are well-known:
\begin{equation}\label{eq:qbin-coinv}
\dqbin{n}{k_1,k_2,\ldots,k_s}{q}=\sum_{w\in\mcR(1^{k_1}2^{k_2}\cdots s^{k_s})}
   q^{\coinv(w)}; \qquad
   \dqbin{n}{k}{q}=\sum_{w\in\mcR(1^k 2^{n-k})} q^{\coinv(w)}. 
\end{equation}
These formulas are often stated with $\coinv(w)$ replaced by the 
\emph{inversion count} $\inv(w)$, which is the number of $i<j$ with $w_i>w_j$.
But the standard proofs using $\inv(w)$ 
(see, for example,~\cite[Chpt. 8]{loehr-comb})
extend at once to $\coinv(w)$ by reversing the natural ordering on $\Z$.

The purpose of this paper is to give bijective proofs of some novel
identities involving the coinversion statistic. These identities 
were recently discovered and proved (algebraically) by 
Ayyer, Mandelshtam, and Martin~\cite{amm-mac} as part of their study
of the combinatorics of the modified Macdonald polynomials $\tilde{H}_{\mu}$.
Macdonald polynomials are not needed in this paper, but the reader may
consult references such as~\cite{hag-mac-conj,HHL-mac,orig-mac,mac-book} 
for more information.

To proceed, we must recall some definitions and results 
from~\cite[Sec. 9]{amm-mac}. 
Fix integers $n\geq 3$, $L>0$, $a_2,\ldots,a_{n-1}\geq 0$, and define
$N=L+a_2+\cdots+a_{n-1}$. For 
$0\leq k\leq L$, define $W_k=\mcR(1^{L-k}2^{a_2}\cdots (n-1)^{a_{n-1}} n^k)$ 
and $W=\bigcup_{k=0}^L W_k$. For any word $w\in W$, let $p_n(w)$ be the
position of the leftmost $n$ in $w$, counting from the left;
let $p_n(w)=\infty$ if no $n$ occurs in $w$. Let $p_1(w)$ be the
position of the rightmost $1$ in $w$, counting from the right end of $w$
and ignoring occurrences of $n$; let $p_1(w)=\infty$ if no $1$ occurs in $w$.
For example, when $n=4$ and $w=3142241324243$, we have $p_4(w)=3$ and 
$p_1(w)=5$. Note that the possible finite values of $p_1(w)$ are
$1,2,\ldots,N-L+1$ since $w$ has $N-k$ symbols (ignoring all copies of $n$)
and $L-k+1$ copies of $1$ must appear to the left of the rightmost $1$ in $w$.
Similarly, the possible finite values of $p_n(w)$ are $1,2,\ldots,N-k+1$.

For $0\leq k\leq L$, define $W_k^>=\{w\in W_k: p_n(w)>p_1(w)\}$
and $W_k^{\leq}=\{w\in W_k: p_n(w)\leq p_1(w)\}$. Since $L>0$,
we have the boundary cases $W_0^>=W_0$, $W_0^{\leq}=\emptyset$,
$W_L^{\leq}=W_L$, and $W_L^>=\emptyset$. 
Define $W^>=\bigcup_{k=0}^L W_k^>$ and $W^{\leq}=\bigcup_{k=0}^L W_k^{\leq}$.
Our main goal is to construct explicit bijective proofs of the following
identities, which are formulas (9.1) through (9.4) of~\cite{amm-mac}.

\begin{theorem}\label{thm:main1}
(a) For $0\leq k\leq L$,
\begin{equation}\label{eq:amm9.1}
 \sum_{w\in W_k^>} q^{\coinv(w)} = \dqbin{N}{L,a_2,\ldots,a_{n-1}}{q}
  \cdot q^k\dqbin{L-1}{k}{q}.
\end{equation}
(b) For $0\leq k\leq L$,
\begin{equation}\label{eq:amm9.2}
 \sum_{w\in W_k^{\leq}} q^{\coinv(w)} = \dqbin{N}{L,a_2,\ldots,a_{n-1}}{q}
  \cdot \dqbin{L-1}{k-1}{q}.
\end{equation}
(c) For $0\leq k\leq L$ and $1\leq i\leq N-L+1$,
\begin{equation}\label{eq:amm9.3}
 \sum_{\substack{w\in W_k^>:\\ p_1(w)=i}} q^{\coinv(w)} =
q^{k+(i-1)L}\dqbin{N-L}{a_2,\ldots,a_{n-1}}{q}\cdot
\dqbin{N-i}{L-k-1,N-L-i+1,k}{q}.
\end{equation}
(d) For $0\leq k\leq L$ and $1\leq j\leq N-k+1$,
\begin{equation}\label{eq:amm9.4}
 \sum_{\substack{w\in W_k^{\leq}:\\ p_n(w)=j}} q^{\coinv(w)} =
 q^{(j-1)L}\dqbin{N-L}{a_2,\ldots,a_{n-1}}{q}\cdot
\dqbin{N-j}{L-k,N-L-j+1,k-1}{q}.
\end{equation}
\end{theorem}

Bijective proofs of parts (a) and (b) of the theorem combine to give
bijections proving
\begin{equation}\label{eq:true-goal}
\sum_{w\in W_k^>} q^{\coinv(w)}=q^k\sum_{w\in W_{k+1}^{\leq}} q^{\coinv(w)}.
\end{equation}
for $0\leq k<L$. These are the crucial bijections the authors of~\cite{amm-mac}
needed to complete their analysis of the quinv statistic in their novel
combinatorial formula for modified Macdonald polynomials.

As was already noted in~\cite{amm-mac}, the general formulas 
in Theorem~\ref{thm:main1} follow easily from the special case where $n=3$.
(This reduction can be done bijectively, as we see later.) 
Given $0<L\leq N$ and $0\leq k\leq L$, let $X_k=\mcR(1^{L-k}2^{N-L}3^k)$,
$X_k^>=\{w\in X_k: p_3(w)>p_1(w)\}$, and 
$X_k^{\leq}=\{w\in X_k: p_3(w)\leq p_1(w)\}$.  For this three-letter case,
we are reduced to proving the following formulas:
\begin{equation}\label{eq:amm9.1'}
\sum_{w\in X_k^>} q^{\coinv(w)} 
= q^k\dqbin{N}{L,N-L}{q}\dqbin{L-1}{L-1-k,k}{q},
\end{equation}
\begin{equation}\label{eq:amm9.2'}
\sum_{w\in X_k^{\leq}} q^{\coinv(w)} 
= \dqbin{N}{L,N-L}{q}\dqbin{L-1}{L-k,k-1}{q},
\end{equation}
\begin{equation}\label{eq:amm9.3'}
\sum_{\substack{w\in X_k^>:\\p_1(w)=i}} q^{\coinv(w)} 
= q^{k+(i-1)L}\dqbin{N-i}{L-k-1,N-L-i+1,k}{q},
\end{equation}
\begin{equation}\label{eq:amm9.4'}
\sum_{\substack{w\in X_k^{\leq}:\\p_3(w)=j}} q^{\coinv(w)} 
= q^{(j-1)L}\dqbin{N-j}{L-k,N-L-j+1,k-1}{q}.
\end{equation}

Our starting point for proving these formulas is the following refinement.

\begin{theorem}\label{thm:main2}
Fix integers $L,N,i,j,k$ with $0<L\leq N$, $0\leq k\leq L$,
$1\leq i\leq N-L+1$, and $1\leq j\leq N-k+1$.
\\ (a) If $0<k<L$, then
\begin{equation}\label{eq:pg53a}
\sum_{\substack{w\in \mcR(1^{L-k}2^{N-L}3^k):\\
 p_1(w)=i\text{ and }p_3(w)=j}} q^{\coinv(w)}
=\dqbin{N-j}{k-1}{q}\dqbin{N-i-k}{L-k-1}{q}q^{(j-1)k+(L-k)(i-1)}. 
\end{equation}
For the words $w$ indexing this sum,
the leftmost $3$ in $w$ occurs to the left of the
rightmost $1$ in $w$ if and only if $i+j+k\leq N+1$.
\\ (b) If $k=0$, then
\begin{equation}\label{eq:pg53b}
\sum_{\substack{w\in \mcR(1^L2^{N-L}):\\
 p_1(w)=i\text{ (and $p_3(w)=\infty$)}}} q^{\coinv(w)}
 =\dqbin{N-i}{L-1}{q}q^{L(i-1)}. 
\end{equation}
\\ (c) If $k=L$, then
\begin{equation}\label{eq:pg53c}
\sum_{\substack{w\in \mcR(2^{N-L}3^L):\\
 p_3(w)=j\text{ (and $p_1(w)=\infty$)}}} q^{\coinv(w)}
=\dqbin{N-j}{L-1}{q}q^{(j-1)L}. 
\end{equation}
\end{theorem}

This paper is organized as follows. Section~\ref{sec:prelim-bij}
gives bijective proofs of some well-known identities for $q$-binomial
coefficients and $q$-multinomial coefficients. 
Section~\ref{sec:bij-thm2} gives a bijective proof of Theorem~\ref{thm:main2}. 
Section~\ref{sec:bij-thm1} gives a bijective proof of Theorem~\ref{thm:main1}.
Section~\ref{sec:example} works out a detailed example illustrating all
the bijections, which leads to a simplified description of the bijective
proof of~\eqref{eq:true-goal}.

\section{Preliminary Bijections}
\label{sec:prelim-bij}

This section describes some preliminary bijections needed to prove the 
main results. We use the following notation and facts about weighted sets. 
A \emph{weighted set} is a set $S$ and a weight function 
$\wt:S\rightarrow\Z_{\geq 0}$. The \emph{generating function}
for this weighted set is $\GF(S)=\sum_{s\in S} q^{\wt(s)}$. 
For any integer $c$ and weighted set $S$, the symbol $q^c S$ refers to the 
set $S$ with shifted weight function $\wt'(s)=\wt(s)+c$.
Note that $\GF(q^c S)=q^c\GF(S)$.
If $A$ and $B$ are weighted sets, then the Cartesian product
$A\times B=\{(a,b):a\in A,b\in B\}$ is a weighted set
with $\wt((a,b))=\wt(a)+\wt(b)$, and $\GF(A\times B)=\GF(A)\GF(B)$. 
If $S_1,\ldots,S_n$ are pairwise disjoint weighted sets, 
then $S_1\cup\cdots\cup S_n$ is a weighted set (using the same weights),
and $\GF(S_1\cup\cdots\cup S_n)=\GF(S_1)+\cdots+\GF(S_n)$.
For weighted sets $S$ and $T$, we write $S\equiv T$ to mean 
there is a weight-preserving bijection between $S$ and $T$;
in this case, $\GF(S)=\GF(T)$. In this paper, we always use $\coinv(w)$ as the
weight of a word $w$ in $\mcR(1^{a_1}\cdots n^{a_n})$.
It follows from~\eqref{eq:qbin-coinv} that $\GF(\mcR(1^{a_1}\cdots n^{a_n}))
 =\tqbin{a_1+\cdots+a_n}{a_1,\ldots,a_n}{q}$.

\subsection{Summation Identity for $q$-Binomial Coefficients}
\label{subsec:sum-qbin}

We will need two bijective versions of the following well-known
identity for summing certain $q$-binomial coefficients.

\begin{prop}\label{prop:FG-bij}
Fix integers $A,B$ with $1\leq B\leq A+1$.
There are weight-preserving bijections 
\[ F:\mcR(0^{A+1-B}3^B)\rightarrow
  \bigcup_{s=0}^{A+1-B}q^{sB}\mcR(0^{A+1-B-s}3^{B-1}), \]
\[ G:\mcR(1^B2^{A+1-B})\rightarrow
  \bigcup_{s=0}^{A+1-B}q^{sB}\mcR(1^{B-1}2^{A+1-B-s}), \]
and therefore
\[ \sum_{s=0}^{A-B+1} \dqbin{A-s}{B-1}{q} q^{sB}=\dqbin{A+1}{B}{q}. \]
\end{prop}
\begin{proof}
Given a word $w\in\mcR(0^{A+1-B}3^B)$, write $w=0^s3w'$ where the displayed 
$3$ is the leftmost $3$ in $w$. Define $F(w)=w'$, which belongs to
$\mcR(0^{A+1-B-s}3^{B-1})$ for some $s$ between $0$ and $A+1-B$. 
Each of the $s$ copies of $0$ at the start of $w$ causes $B$ coinversions with
the $3$s later in $w$. These coinversions are not present in $w'$, but
all other coinversions in $w$ and $w'$ are the same. Thus,
$\coinv(w)=sB+\coinv(w')$, so that $F$ preserves weights.
$F$ is a bijection with inverse $F^{-1}(w')=0^s3w'$. 
When computing the inverse, we can deduce $s$ from $w'$ by
counting the $0$s in $w'$ (since $A$ and $B$ are fixed and known).
In detail, writing $n_0(w')$ for the number of $0$s in $w'$,
we have $s=A+1-B-n_0(w')$.

The bijection $G$ is defined and analyzed similarly: 
given $w\in\mcR(1^B2^{A+1-B})$, write $w=w'12^s$ where the
displayed $1$ is the rightmost $1$ in $w$, and let $G(w)=w'$. 
Each $2$ at the end of $w$ causes coinversions with all $B$ copies of $1$
appearing earlier, so passing from $w$ to $w'$ reduces $\coinv(w)$ by $sB$.
So $G$ is a weight-preserving bijection.
\end{proof}

\subsection{Factorization Identities for $q$-Multinomial Coefficients}
\label{subsec:factor-qmul}

\begin{prop}\label{prop:H-bij}
Fix integers $A,B,C\geq 0$. There is a weight-preserving bijection
\[ H:\mcR(1^A 2^B 3^C)\rightarrow\mcR(0^{A+B} 3^C)\times\mcR(1^A 2^B), \]
and therefore 
\[ \dqbin{A+B+C}{A,B,C}{q}=\dqbin{A+B+C}{A+B,C}{q}\dqbin{A+B}{A,B}{q}. \]
\end{prop}
\begin{proof}
Let $H$ map $v\in\mcR(1^A 2^B 3^C)$ to $(y,z)\in\mcR(0^{A+B} 3^C)\times
\mcR(1^A 2^B)$, where $y$ is obtained from $v$ by replacing each occurrence
of $1$ or $2$ by $0$, and $z$ is obtained from $v$ by erasing all $3$s.
For example, $H(231132)=(030030,2112)$.
It is routine to check that $\coinv(v)=\coinv(y)+\coinv(z)=\wt((y,z))$,
so $H$ is weight-preserving. We invert $H$ by using the $1$s and $2$s in $z$
(reading left to right) to replace the $A+B$ copies of $0$ in $y$.
For example, $H^{-1}(303000,1212)=313212$.
\end{proof}

The same relabeling idea gives bijective proofs of related identities for
$q$-multinomial coefficients. 
For example, assuming $a_2+\cdots+a_{n-1}=N-L$, we get
\begin{equation}\label{eq:relabel2}
 \mcR(1^{L-k}2^{a_2}\cdots (n-1)^{a_{n-1}}n^k)\equiv
 \mcR(1^{L-k}2^{N-L}3^k)\times\mcR(2^{a_2}3^{a_3}\cdots (n-1)^{a_{n-1}}) 
\end{equation}
by mapping $v$ to $(y,z)$, where $y$ is $v$ with the middle letters
$2,\ldots,n-1$ all relabeled as $2$ and the biggest letter $n$ relabeled as $3$,
and $z$ is $v$ with all copies of $1$ and $n$ erased. We also get
\begin{equation}\label{eq:relabel3}
 \mcR(0^L3^{N-L})\times\mcR(2^{a_2}3^{a_3}\cdots (n-1)^{a_{n-1}})
 \equiv \mcR(1^L2^{a_2}3^{a_3}\cdots (n-1)^{a_{n-1}}) 
\end{equation}
by mapping $(y,z)$ to $v$, where $v$ is $y$ with each $0$ replaced by $1$
and the subword $3^{N-L}$ replaced by the word $z$. It is routine to check
that these maps are weight-preserving bijections.

\subsection{Symmetry Identities for $q$-Multinomial Coefficients}
\label{subsec:symm-qtri}

\begin{prop}\label{prop:K-bij}
For all integers $A,B,C\geq 0$, there is a weight-preserving bijection
\[ K:\mcR(1^A 2^B 3^C)\rightarrow \mcR(1^A 2^C 3^B). \]
\end{prop}
\begin{proof}
Let $K$ act on $v\in\mcR(1^A 2^B 3^C)$ by replacing each $2$ by $3$
and each $3$ by $2$ in $v$, then reversing the subword of $2$s and $3$s 
in $v$. For example, $K(2311323331)=2211232231$. We invert $K$ by
performing the same actions on the output word. It is routine to check
that $K$ preserves coinversions.
\end{proof}

More generally, if $b_1,\ldots,b_n$ is any permutation of
$a_1,\ldots,a_n$, we have the algebraically obvious symmetry property
\begin{equation}\label{eq:multinom-symm}
 \dqbin{a_1+\cdots+a_n}{a_1,\ldots,a_n}{q}
  =\dqbin{a_1+\cdots+a_n}{b_1,\ldots,b_n}{q}.
\end{equation}
The idea in the preceding proof generalizes at once to
give a bijection interchanging the frequencies of any two adjacent letters.
Composing several bijections of this form, we can transform
the initial frequencies ($a_i$ copies of $i$ for all $i$)
to the final frequencies ($b_i$ copies of $i$ for all $i$).
This gives a bijective proof of~\eqref{eq:multinom-symm}
by showing $\mcR(1^{a_1}2^{a_2}\cdots n^{a_n})\equiv
            \mcR(1^{b_1}2^{b_2}\cdots n^{b_n})$.

\section{Bijective Proof of Theorem~\ref{thm:main2}}
\label{sec:bij-thm2}

To prove Theorem~\ref{thm:main2}(a),
let $L,N,i,j,k$ be integers with $0<L\leq N$, $0<k<L$,
$1\leq i\leq N-L+1$, and $1\leq j\leq N-k+1$.
It suffices to define a weight-preserving bijection $P=P_{i,j}$ 
mapping the domain
\begin{equation}\label{eq:main2-dom}
\mcR(0^{N-j-k+1}3^{k-1})\times\mcR(1^{L-k-1}2^{N-L-i+1})
q^{(j-1)k+(L-k)(i-1)} 
\end{equation}
one-to-one onto the codomain
\[ \{w\in\mcR(1^{L-k}2^{N-L}3^k): p_1(w)=i\mbox{ and }p_3(w)=j\}. \]
Given an input $(y,z)$ in the domain of $P$, we build $w=P(y,z)$ as follows.
Start with $N$ empty slots for the $N$ symbols in $w$.
Put a $3$ in slot $j$; we refer to this $3$ as L3 (the leftmost $3$ in $w$).  
There are $N-j$ slots to the right of position $j$, 
and the remaining $k-1$ copies of $3$ must go in these slots to ensure
that $p_3(w)=j$. Place the word $y$ in these $N-j$ slots, regarding a $0$
in $y$ as a slot in $w$ that still remains empty for now. Next, visit the empty
slots in $w$ from right to left, placing $i-1$ copies of $2$ followed by
a $1$, which ensures that $p_1(w)=i$. We refer to this copy of $1$ as R1
(the rightmost $1$ in $w$).  Finally, fill the remaining empty
slots in $w$ (to the left of the $1$ just placed) with the remaining
$L-k-1$ copies of $1$ and the remaining $N-L-i+1$ copies of $2$. Do this
by reading $z$ (left to right) and filling the empty slots (left to right)
using the symbols in $z$. It is routine to check that this procedure is
invertible, so $P$ is a bijection. 

Suppose R1 is placed to the left of L3 in $w$. This forces
R1 to be the leftmost symbol in the collection $\mcC$ consisting 
of the $k-1$ copies of $3$ to the right of L3, the $i-1$ copies of
$2$ to the right of R1, and R1 itself. 
Since there are $N-j$ available slots to the right of L3, 
we must have $(k-1)+(i-1)+1>N-j$, so $i+j+k>N+1$.
Conversely, if $i+j+k>N+1$, then the $N-j$ slots to the right of L3
cannot accommodate all symbols in $\mcC$, which forces R1 to be
placed to the left of L3 in $w$.

To see that $P$ preserves weights, we show that
$\coinv(w)=\coinv(y)+\coinv(z)+(j-1)k+(L-k)(i-1)$.
Note that every $0$ in $y$ is placed in $w$ to the right of L3
and eventually gets relabeled as a $1$ or $2$. Thus, $\coinv(y)$
counts all coinversions in $w$ involving a $1$ or $2$ to the right of L3
followed by a $3$ to the right of L3. Similarly, $\coinv(z)$
counts all coinversions in $w$ involving a $1$ to the left of R1
followed by a $2$ to the left of R1. We finish counting the coinversions
of $w$ as follows. First, each of the $j-1$ symbols to the left of L3 
(which must be $1$ or $2$)
causes a coinversion with each of the $k$ copies of $3$ in $w$, giving
$(j-1)k$ coinversions. Second, each of the $L-k$ copies of $1$ in $w$
causes a coinversion with each of the $i-1$ copies of $2$ to the right 
of R1, giving $(L-k)(i-1)$ coinversions. This explains the weight-shifting
factor in the domain~\eqref{eq:main2-dom}.

For example, let $N=13$, $L=6$, $k=4$, $i=5$, $j=3$,
$y=0030003030$, and $z=2122$. Then $P(y,z)=w=2132231223232$
where $p_1(w)=5$ and $p_3(w)=3$. Note that $\coinv(y)=13$, $\coinv(z)=2$,
and $\coinv(w)=31=13+2+2\cdot 4+2\cdot 4$.

We can prove parts~(b) and~(c) of Theorem~\ref{thm:main2} by using
degenerate versions of the bijection used to prove~(a). 
For the $k=0$ case, we define a bijection
\begin{equation}\label{eq:k=0case}
 \mcR(1^{L-1}2^{N-i-L+1})q^{L(i-1)}\rightarrow
  \{w\in\mcR(1^L2^{N-L}): p_1(w)=i\} 
\end{equation}
by mapping $z$ in the domain to $w=z12^{i-1}$ in the codomain.
The extra $q$-power appears since each of the $L$ copies of $1$ in $w$
causes a coinversion with each of the $i-1$ copies of $2$ added at the end.
For the $k=L$ case, we define a bijection
\begin{equation}\label{eq:k=Lcase}
 \mcR(2^{N-L-j+1}3^{L-1})q^{L(j-1)}\rightarrow
 \{w\in\mcR(2^{N-L}3^L): p_3(w)=j\} 
\end{equation}
by mapping $y$ in the domain to $w=2^{j-1}3y$ in the codomain.
The extra $q$-power appears since each of the $j-1$ copies of $2$ at
the start of $w$ causes a coinversion with all $L$ copies of $3$ in $w$.

\section{Bijective Proof of Theorem~\ref{thm:main1}}
\label{sec:bij-thm1}

\subsection{Proof of~\eqref{eq:amm9.3'}}
\label{subsec:amm9.3'}

Fix $i$ with $1\leq i\leq N-L+1$. To prove~\eqref{eq:amm9.3'}
when $0<k<L$, we must define a weight-preserving bijection 
\begin{equation}\label{eq:bij9.3'}
 \{w\in\mcR(1^{L-k}2^{N-L}3^k): p_3(w)>p_1(w)=i\}
\rightarrow \mcR(1^{L-k-1}2^{N-L-i+1}3^k)q^{k+(i-1)L}. 
\end{equation}
Take the disjoint union of the bijections $P_{i,j}^{-1}$ as 
$j$ ranges over possible values larger than $i$. This maps the
domain in~\eqref{eq:bij9.3'} to the disjoint union
\[ \bigcup_{j=i+1}^{N-k+1} \mcR(0^{N-j-k+1}3^{k-1})\times
 \mcR(1^{L-k-1}2^{N-L-i+1}) q^{(j-1-i)k+L(i-1)+k}. \]
Letting $s=j-i-1$, we can write this union as
\[ \left(\bigcup_{s=0}^{N-k-i} \mcR(0^{N-i-k-s}3^{k-1})q^{sk}\right)
 \times \mcR(1^{L-k-1}2^{N-L-i+1})q^{k+L(i-1)}. \]
The parenthesized piece is the codomain of the bijection $F$
in Proposition~\ref{prop:FG-bij}, taking $B=k$ and $A=N-i-1$. 
Applying $F^{-1}$ to this piece, we get a bijection to
\[ \mcR(0^{N-i-k}3^k)\times \mcR(1^{L-k-1}2^{N-L-i+1})q^{k+L(i-1)}. \]
Applying $H^{-1}$ from Proposition~\ref{prop:H-bij}, we reach 
\[ \mcR(1^{L-k-1}2^{N-L-i+1}3^k) q^{k+L(i-1)}, \]
which is the codomain in~\eqref{eq:bij9.3'}.
When $k=0$,~\eqref{eq:bij9.3'} reduces to
\[ \{w\in\mcR(1^L2^{N-L}): p_1(w)=i\}\rightarrow
 \mcR(1^{L-1}2^{N-L-i+1})q^{L(i-1)}, \] which is the inverse
of the bijection~\eqref{eq:k=0case}. 
When $k=L$, both sides of~\eqref{eq:amm9.3'} are $0$.

\subsection{Proof of~\eqref{eq:amm9.4'}}
\label{subsec:amm9.4'}

Fix $j$ with $1\leq j\leq N-k+1$. To prove~\eqref{eq:amm9.4'}
when $0<k<L$, we must define a weight-preserving bijection 
\begin{equation}\label{eq:bij9.4'}
 \{w\in\mcR(1^{L-k}2^{N-L}3^k): j=p_3(w)\leq p_1(w)\}
\rightarrow \mcR(1^{L-k}2^{N-L-j+1}3^{k-1})q^{L(j-1)}. 
\end{equation}
Take the disjoint union of the bijections $P_{i,j}^{-1}$ as 
$i$ ranges over its possible values that are at least $j$. This maps the
domain in~\eqref{eq:bij9.4'} to the disjoint union
\[ \bigcup_{i=j}^{N-L+1} \mcR(0^{N-j-k+1}3^{k-1})\times
 \mcR(1^{L-k-1}2^{N-L-i+1}) q^{(j-1)k+(L-k)(i-1)}. \]
Letting $s=i-j$, we can write this union as
\[ q^{L(j-1)}\mcR(0^{N-j-k+1}3^{k-1})\times\left(\bigcup_{s=0}^{N-L-j+1} 
 \mcR(1^{L-k-1}2^{N-L-s-j+1})q^{s(L-k)} \right). \]
The parenthesized piece is the codomain of the bijection $G$
in Proposition~\ref{prop:FG-bij}, taking $B=L-k$ and $A=N-k-j$. 
Applying $G^{-1}$ to this piece, we get a bijection to
\begin{equation}\label{eq:recover-j}
 q^{L(j-1)}\mcR(0^{N-j-k+1}3^{k-1})\times\mcR(1^{L-k}2^{N-L-j+1}). 
\end{equation}
Applying $H^{-1}$ from Proposition~\ref{prop:H-bij}, we reach 
\[ \mcR(1^{L-k}2^{N-L-j+1}3^{k-1}) q^{L(j-1)}, \]
which is the codomain in~\eqref{eq:bij9.4'}.
When $k=L$,~\eqref{eq:bij9.4'} reduces to
\[ \{w\in\mcR(2^{N-L}3^L): p_3(w)=j\}
\rightarrow \mcR(2^{N-L-j+1}3^{L-1})q^{L(j-1)},  \]
which is the inverse of the bijection~\eqref{eq:k=Lcase}. 
When $k=0$, both sides of~\eqref{eq:amm9.4'} are $0$.

\subsection{Proof of~\eqref{eq:amm9.1'}}
\label{subsec:amm9.1'}

To prove~\eqref{eq:amm9.1'} for fixed $k$, we build a weight-preserving
bijection
\begin{equation}\label{eq:bij9.1'}
 \{w\in\mcR(1^{L-k}2^{N-L}3^k): p_3(w)>p_1(w) \} \rightarrow
 q^k\mcR(0^L 3^{N-L})\times \mcR(1^{L-k-1}2^k).
\end{equation}
Take the disjoint union of the bijections~\eqref{eq:bij9.3'} over all possible
$i$. This maps the domain in~\eqref{eq:bij9.1'} to the disjoint union
\[ \bigcup_{i=1}^{N-L+1} \mcR(1^{L-k-1}2^{N-L-i+1}3^k)q^{L(i-1)+k}. \]
Use the bijection $K$ of Proposition~\ref{prop:K-bij} to interchange
the frequencies of $2$s and $3$s, which yields
\[ \bigcup_{i=1}^{N-L+1} \mcR(1^{L-k-1}2^k3^{N-L-i+1})q^{L(i-1)+k}. \]
Next use the bijection $H$ of Proposition~\ref{prop:H-bij} to reach
\[  \left(\bigcup_{i=1}^{N-L+1} \mcR(0^{L-1}3^{N-L-i+1})q^{L(i-1)}\right)
  \times q^k\mcR(1^{L-k-1}2^k). \]
To finish, use $G^{-1}$ from Proposition~\ref{prop:FG-bij} 
(taking $s=i-1$, $A=N-1$, $B=L$ and replacing each $1$ by $0$
and each $2$ by $3$) to reach
\[ q^k\mcR(0^L 3^{N-L})\times\mcR(1^{L-k-1}2^k). \]

\subsection{Proof of~\eqref{eq:amm9.2'}}
\label{subsec:amm9.2'}

To prove~\eqref{eq:amm9.2'} for fixed $k$, we build a weight-preserving
bijection
\begin{equation}\label{eq:bij9.2'}
 \{w\in\mcR(1^{L-k}2^{N-L}3^k): p_3(w)\leq p_1(w) \} \rightarrow
 \mcR(0^L 3^{N-L})\times \mcR(1^{L-k}2^{k-1}).
\end{equation}
Take the disjoint union of the bijections~\eqref{eq:bij9.4'} over all possible
$j$. (Since $j\leq p_1(w)$ here, the upper limit for $j$ is $N-L+1$.)
We thereby map the domain in~\eqref{eq:bij9.2'} to the disjoint union
\[ \bigcup_{j=1}^{N-L+1} \mcR(1^{L-k}2^{N-L-j+1}3^{k-1})q^{L(j-1)}. \]
Use the bijection $K$ of Proposition~\ref{prop:K-bij} to interchange
the frequencies of $2$s and $3$s, which yields
\[ \bigcup_{j=1}^{N-L+1} \mcR(1^{L-k}2^{k-1}3^{N-L-j+1})q^{L(j-1)}. \]
Next use the bijection $H$ of Proposition~\ref{prop:H-bij} to reach
\[  \left(\bigcup_{j=1}^{N-L+1} \mcR(0^{L-1}3^{N-L-j+1})q^{L(j-1)}\right)
  \times \mcR(1^{L-k}2^{k-1}). \]
To finish, use $G^{-1}$ from Proposition~\ref{prop:FG-bij} 
(taking $s=j-1$, $A=N-1$, $B=L$ and replacing each $1$ by $0$
and each $2$ by $3$) to reach
\[ \mcR(0^L 3^{N-L})\times\mcR(1^{L-k}2^{k-1}). \]

\subsection{Proof of~\eqref{eq:amm9.3} and~\eqref{eq:amm9.4}}
\label{subsec:prove-9.3,9.4}  

To prove~\eqref{eq:amm9.3}, we need a weight-preserving bijection
\begin{equation}\label{eq:bij9.3}
 \{w\in W_k^{>}: p_1(w)=i\}\rightarrow
 q^{k+(i-1)L}\mcR(1^{L-k-1} 2^{N-L-i+1} 3^k)
 \times\mcR(2^{a_2}\cdots (n-1)^{a_{n-1}}). 
\end{equation}
Restricting the bijection~\eqref{eq:relabel2} to
the domain of~\eqref{eq:bij9.3}, we get a bijection mapping that domain to
\[ \{w'\in X_k^>: p_1(w')=i\}\times\mcR(2^{a_2}\cdots (n-1)^{a_{n-1}}). \]
Now, apply~\eqref{eq:bij9.3'} to the first factor
to reach the codomain of~\eqref{eq:bij9.3}.
We prove~\eqref{eq:amm9.4} in the same way, using~\eqref{eq:bij9.4'}.

\subsection{Proof of~\eqref{eq:amm9.1} and~\eqref{eq:amm9.2}}
\label{subsec:prove-9.1,9.2}  

To prove~\eqref{eq:amm9.1}, we need a weight-preserving bijection
\begin{equation}\label{eq:bij9.1}
 W_k^{>}\rightarrow q^k\mcR(1^L2^{a_2}\cdots(n-1)^{a_{n-1}})
 \times\mcR(1^{L-1-k}2^k).
\end{equation}
Restricting the bijection~\eqref{eq:relabel2} to
the domain of~\eqref{eq:bij9.1}, we get a bijection mapping that domain to
\[ X_k^{>}\times\mcR(2^{a_2}\cdots (n-1)^{a_{n-1}}). \]
Apply~\eqref{eq:bij9.1'} to the first factor to reach
\[ q^k\mcR(0^L 3^{N-L})\times \mcR(1^{L-k-1}2^k)\times 
 \mcR(2^{a_2}\cdots (n-1)^{a_{n-1}}). \]
Finally, apply bijection~\eqref{eq:relabel3} to the first and third factors
in this Cartesian product to obtain
\[ q^k\mcR(1^L 2^{a_2}\cdots (n-1)^{a_{n-1}})\times \mcR(1^{L-k-1}2^k). \]
We prove~\eqref{eq:amm9.2} in the same way, using~\eqref{eq:bij9.2'}.

\section{A Detailed Example and the Proof of~\eqref{eq:true-goal}.}
\label{sec:example}

In this section, we start with a specific $w\in W_k^>$ and trace through all 
the bijections in the proof of~\eqref{eq:amm9.1}
to find the image of $w$ in the codomain of \eqref{eq:bij9.1}.
We continue by dropping the weight-shift factor $q^k$, replacing $k$ by $k+1$,
and tracing the proof of~\eqref{eq:amm9.2} backwards from this codomain to
get $w'\in W_{k+1}^{\leq}$ such that $\coinv(w)=\coinv(w')+k$. 
Some intermediate bijections cancel out in this two-step process, 
leading us to a simpler bijective proof of~\eqref{eq:true-goal}.

\subsection{Mapping $w\in W_k^>$ to the Intermediate Object}
\label{subsec:ex-part1}

Let $w=3112443214243\in\mcR(1^3 2^3 3^3 4^4)$, so
$n=4$, $k=4$, $L=7$, $a_2=3$, $a_3=3$, $N=13$, $\coinv(w)=39$,
$p_4(w)=5$, $p_1(w)=3$, and $w\in W_4^>$. We follow the proof
of~\eqref{eq:amm9.1} to send $w$ to an intermediate object 
in the codomain of the map~\eqref{eq:bij9.1}.

\begin{itemize}
\item \emph{Step~1.} Apply bijection~\eqref{eq:relabel2} to convert
$w$ to the pair $(y,z)$, where $y=2112332213232$ and $z=323223$.
Note $y\in X_4^{>}$ with $p_3(y)=5>3=p_1(y)$, $\coinv(y)=35$,
and $\coinv(z)=4$. The next six steps apply to $y$ alone.
\item \emph{Step~2.} Apply bijection $P_{i,j}^{-1}$ to $y$,
where $i=p_1(y)=3$ and $j=p_3(y)=5$. We get the pair
$(30003030\cdot q^4,211222\cdot q^{18})$, where the extra $q$-powers
indicate weight-shifting amounts for each component word.
\item \emph{Step~3.} Apply bijection $F^{-1}$ (Proposition~\ref{prop:FG-bij})
with $A=9$, $B=4$, $s=j-i-1=1$ to the first component, producing
$(0330003030,211222\cdot q^{18})$.
\item \emph{Step~4.} Apply bijection $H^{-1}$ (Proposition~\ref{prop:H-bij})
to change this pair to $2331123232\cdot q^{18}$.
\item \emph{Step~5.} Apply bijection $K$ (Proposition~\ref{prop:K-bij})
to reach $3231123223\cdot q^{18}$.
\item \emph{Step~6.} Apply bijection $H$ to get the pair
$(3030003003\cdot q^{14},211222\cdot q^{4})$.
\item \emph{Step~7.} Apply bijection $G^{-1}$ (with $i=3$, $s=i-1=2$, $A=12$,
$B=7$) to get $$(3030003003033,211222\cdot q^{4}).$$
\item \emph{Step~8.} Apply bijection~\eqref{eq:relabel3} to combine
the first component here with $z$. We thereby reach the intermediate object
\[ q^4(3121113112123,211222)\in q^k\mcR(1^7 2^3 3^3)\times\mcR(1^2 2^4). \]
This object has weight $4+29+6=39=\coinv(w)$.
\end{itemize}

\subsection{Mapping the Intermediate Object to $w'\in W_{k+1}^{\leq}$}
\label{subsec:ex-part2}

We continue operating on the intermediate object, dropping the $q^4$ 
shift and working backwards through the proof of~\eqref{eq:amm9.2},
taking $k=5$ now. We discover that the first five steps undo
the last five steps of the previous algorithm:

\begin{itemize}
\item\emph{Step~$8'$.} Apply the inverse of bijection~\eqref{eq:relabel3}
to produce the triple $(3030003003033,211222,z=323223)$. We save $z$ for later
and keep acting on the first two components.
\item\emph{Step~$7'$.} Apply $G$ to the first component to
get $(3030003003\cdot q^{14},211222)$.
\item\emph{Step~$6'$.} Apply $H^{-1}$ to this pair to get
$3231123223\cdot q^{14}$.
\item\emph{Step~$5'$.} Apply $K^{-1}$ to get
$2331123232\cdot q^{14}$.
\item\emph{Step~$4'$.} Apply $H$ to get
$(0330003030,211222)\cdot q^{14}$. Comparing to~\eqref{eq:recover-j},
we see $j=3$ now.
\item\emph{Step~$3'$.} Apply $G$ (with $j=3$, $s=3$, $A=5$, $B=2$, $i=s+j=6$) 
to get $(0330003030\cdot q^{14},21\cdot q^6)$.
\item\emph{Step~$2'$.} Apply $P_{6,3}$ to get
$2131332223232$.
\item\emph{Step~$1'$.} Apply the inverse of~\eqref{eq:relabel2} to the pair
$(2131332223232,z=323223)$ to get the final output
$w'=3141442324243$. Note $w'\in W_5^{\leq}$ has $p_4(w')=3\leq 6=p_1(w')$,
and $\coinv(w')=35=\coinv(w)-4$.
\end{itemize}

\subsection{Simplified Bijection Proving $W_k^>\equiv q^k W_{k+1}^{\leq}$}
\label{subsec:bij-true-goal}

To prove~\eqref{eq:true-goal} bijectively, we need a weight-preserving
bijection from $W_k^>$ to $q^k W_{k+1}^{\leq}$ for $0\leq k<L$. 
The cancellation of the actions in Steps 3 through 8 of the example holds 
in general. Thus we arrive at the following simplified description of how 
the required bijection acts on $w\in W_k^>$:

\begin{itemize}
\item[(a)] Use bijection~\eqref{eq:relabel2} to convert $w$ to $(y,z)$
where $y\in X_k^>$.
\item[(b)] Apply $P_{i,j}^{-1}$ to $y$, where $i=p_1(y)$ and $j=p_3(y)$,
so that $(y,z)$ becomes $(u,v,z)$.
\item[(c)] Apply $F^{-1}$ to $u$ by prepending $0^s3$, where $s=N-i-k-n_0(u)$,
to get $(u',v,z)$.
\item[(d)] Apply $G$ to $v$, which removes a suffix $12^{s'}$ and leaves
us with $(u',v',z)$.  Hereafter, we use new variables 
$k'=k+1$, $j'=N-k'+1-n_0(u')=i$, and $i'=s'+j'\geq j'$.
\item[(e)] Apply $P_{i',j'}$ to $(u',v')$ to change $(u',v',z)$ to $(y',z)$,
 where $y'\in X_{k'}^{\leq}=X_{k+1}^{\leq}$.
\item[(f)] Use the inverse of bijection~\eqref{eq:relabel2} to map  
$(y',z)$ to the final output word $w'\in W_{k'}^{\leq}=W_{k+1}^{\leq}$.
\end{itemize}

In the boundary cases where $k=0$ or $k'=L$, we modify these steps as follows.
When $k=0$, replace $P_{i,j}^{-1}$ in step~(b) by the inverse of 
bijection~\eqref{eq:k=0case}, which transforms $y$ to $v$;
and replace step~(c) by setting $u'=0^{N-i}$. 
For any $k$ in the range $0\leq k<L$, steps~(b) and (c) send
a word $y\in X_k^>$ with $p_1(y)=i$ to a pair $(u',v)$ in
\begin{equation}\label{eq:intmed1}
 \mcR(0^{N-i-k}3^k)\times\mcR(1^{L-k-1}2^{N-L-i+1})q^{k+L(i-1)}.
\end{equation}

When $k'=L$, we must have $v=2^{N-L-j'+1}$. Here, we modify step~(d) by
discarding $v$ from $(u',v,z)$ and replacing all $0$s in $u'$ by $2$s.
In step~(e), we replace $P_{i',j'}$ by the bijection~\eqref{eq:k=Lcase}, 
which transforms the modified $u'$ to a word $y'\in X_L^{\leq}$.
For any $k'$ in the range $0<k'\leq L$, doing the inverse of step~(e)
followed by the inverse of step~(d) sends a word $y'\in X_{k'}^{\leq}$
with $p_3(y')=j'$ to a pair $(u',v)$ in 
\begin{equation}\label{eq:intmed2}
 \mcR(0^{N-j'-k'+1}3^{k'-1})\times\mcR(1^{L-k'}2^{N-L-j'+1})q^{L(j'-1)}.
\end{equation}
Since $k'=k+1$ and $j'=i$, the intermediate collections~\eqref{eq:intmed1}
and~\eqref{eq:intmed2} match after shifting the latter by $q^k$, as needed.

 
\end{document}